\documentclass[12pt,letterpaper]{article}
\usepackage[utf8]{inputenc}
\usepackage{colortbl,color}
\usepackage{amsmath,amsfonts,amssymb,amsthm,amstext}
\usepackage{graphicx,mathtools,array,float}
\usepackage[margin=2.5cm]{geometry}
\usepackage[affil-it]{authblk}

\newtheorem{theorem}{Theorem}

\newtheorem{lemma}[theorem]{Lemma}

\definecolor{olivegreen}{rgb}{0.05,0.7,0.1}

\title{A note on  maximal plane subgraphs of the complete twisted graph  containing perfect matchings}

\author{Elsa Oma\~na-Pulido\thanks{eomagna@gmail.com} }

\author{Eduardo Rivera-Campo\thanks{erc@xanum.uam.mx} \thanks{Corresponding author}}

\affil{Departamento de Matem\'aticas\\Universidad Aut\'onoma Metropolitana - Iztapalapa\\Ciudad de M\'exico, M\'exico}

\date{}

\begin{document}

\maketitle

\begin{abstract}

The twisted graph $T_{n}$ is a drawing of the complete graph with $n$
vertices $v_{1},v_{2},\ldots ,v_{n}$ in which two edges $v_{i}v_{j}$
($i<j$) and $v_{s}v_{t}$ ($s<t$) cross if and only if
$i<s<t<j$ or $s<i<j<t$. We  show that for any maximal plane subgraphs $S$ and $R$ of $T_{n}$, each containing at least one  perfect matching, there is a sequence $S=F_0, F_1, \ldots, F_t=R$ of maximal plane subgraphs of $T_n$, also containing perfect matchings, such that for $i=0,1, \ldots, t-1$, $F_{i+1}$ can be obtained from $F_{i}$ by a single edge exchange i.e. $F_{i+1} = ((F_{i} - e) + f)$, where $e$ is an edge of $F_{i}$ not in $F_{i+1}$ and $f$ is an edge of $F_{i+1}$ not in $F_{i}$.

\bigskip


\end{abstract}

\section{Introduction}
\label{sec:intro}

Consider a set of points $P$ in the Euclidean plane. A \emph{topological graph}
with vertex set $P$ is a simple graph whose edges are drawn in the plane 
as Jordan curves in such a way that any two edges have at most one
point in common and any pair of non-incident edges with a point in common must cross at that point. A topological graph $G$ is a \emph{geometric graph} if the vertices of $G$ are in general position and all edges of $G$ are straight line segments. A geometric graph is a  \emph{convex geometric graph} if the set of vertices is in convex position.

Two topological graphs $G$ and $G^{\prime }$ with vertex sets $P$
and $P^{\prime }$ and edge sets $E$ and $E'$, respectively, are said to be \emph{weakly isomorphic} if there is a graph isomorphism $\alpha:P\rightarrow P^{\prime }$ such that for all $x, y, u, v \in P$, two edges $xy$ and $uv$ cross each other in $G$ if and only if the corresponding edges $\alpha \left( x\right) \alpha \left( y\right) $ and $\alpha \left(
u\right) \alpha \left( v\right) $ cross each other in $G^{\prime }$.

The \emph{twisted graph} $T_{n}$ is a complete topological graph
with $n$ vertices $v_{1},v_{2},\allowbreak \ldots ,v_{n}$ in which
two edges $v_{i}v_{j}$ ($i<j$) and $v_{s}v_{t}$ ($s<t$) cross each
other if and only either $i<s<t<j$ or $s<i<j<t$. See Fig. \ref{T_6}

\begin{figure}[h] 
\begin{center}
\includegraphics[width=.6\textwidth]{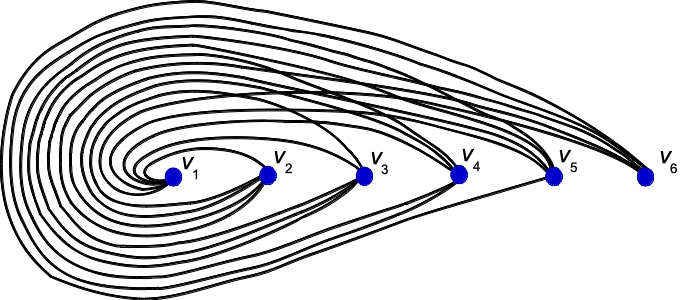}
\label{T_6} 
\caption{\small The twisted graph $T_6$.}
\end{center}
\end{figure}

Harborth and Mengersen \cite{HM} introduced twisted graphs as examples of complete topological graphs containing no
subgraphs weakly isomorphic to the complete convex
geometric graph $C_{5}$ on 5 vertices. Pach et al.
\cite{PST} proved that if $G$ is a complete topological graph with $n$
vertices, then $G$ contains a complete subgraph with $m\geq c\log
^{1/8}n$ vertices which is weakly isomorphic to either the complete
convex geometric graph $C_{m}$ or to the twisted graph $T_{m}$. Recently, Suk and Zeng \cite{SZ} improved this bound to $ \log^{1/4-o(1)}n$. 
Twisted graphs have been generalized by Aichholzer et al. \cite{AGTVW}.

Given a complete topological graph $F_n$ with $n$ vertices. A plane
subgraph $G$ of $F_n$ is a  \emph{maximal plane subgraph} of $F_{n}$ 
if for each edge $uv$ of $F_{n}$, not in $G$,
there is an edge $xy$ of  $G$ that crosses $uv$. The \emph{max
graph } $MP\left( F_{n}\right) $ of $F_{n}$ is the abstract graph
whose vertices are the maximal plane subgraphs of
$F_{n}$ in which two graphs $G$ and $ H $ are adjacent  if there are edges $g$ of $G$ and $h$ of $H$
such that $H=\left( G-g\right) +h$. In the case where $F_n$ 
is a complete geometric graph with vertex set $R$, the graph $MP\left(F_{n}\right) $ is called the \emph{graph 
of triangulations} of $R$ and is denoted by $\mathcal{T}(R)$. Lawson \cite{L} proved that for each set  of points $R$ in general position in the plane, the graph $\mathcal{T}(R)$ is connected. Later, Oma\~na-Pulido and Rivera-Campo \cite{OR} showed that $MP\left(T_{n}\right)$ is also connected.

Houle  et al.  \cite{HHNR} showed that if $R$ has an even number of points,  the subgraph $\mathcal{T_M}(R)$ of $\mathcal{T}(R)$, induced by those triangulations of $R$ that contain perfect matchings, is also connected.   In this note we use some of the ideas of  Houle  et al.  \cite{HHNR} and of  Oma\~na-Pulido and Rivera-Campo \cite{OR} to show that for any maximal plane subgraphs $R$ and $S$ of $T_{2m}$, each one containing at least one perfect matching, there is a sequence $R=F_0, F_1, \ldots, F_t=S$ of maximal plane subgraphs of $T_{2m}$, also containing perfect matchings, such that for $i=0,1, \ldots, t-1$, $F_{i+1}$ can be obtained from $F_{i}$ by a single edge exchange.

Throughout the paper we assume $T_{n}$ is represented as in Fig. 1 with  its vertices denoted by $v_{1},v_{2},\ldots ,v_{n}$ 
  from left to right.

\section{Preliminary results}
\label{sec:prel}

The \emph{matching graph} $\mathcal{M}(T_{2m})$ of $T_{2m}$ is the
abstract graph in which the vertices are the plane perfect
matchings of $T_{2m}$ and two matchings $L$ and  $N$ are adjacent
if the symmetric difference $L \bigtriangleup N$ is a plane cycle
with four edges.  Oma\~na-Pulido and Rivera-Campo \cite{OR}  proved that $\mathcal{M}(T_{2m})$ is always connected.

\begin{theorem} \label{teomatchings} \cite{OR}
For each positive integer $m$, the graph $\mathcal{M}\left( T_{2m}\right) $  is
connected.
\end{theorem}

Let $n \geq 3$ be an integer and $R$ be a maximal plane subgraph of $T_n$. For $i = 1, 2, \ldots, n-1$, let $N_R^+(v_i) = \{v_j : v_iv_j \in E(R) \text{ with } j > i \}$, $d_R^+(v_i) = \vert N_R^+(v_i) \vert$, $min_R(v_i) = min\{j : v_j \in N_R^+(v_i)\}$ and $max_R(v_i) = max\{j : v_j \in N_R^+(v_i)\}$. 

\begin{lemma}
    Let $n\geq 3$ be an integer and $R$ be a maximal plane subgraph of $T_n$. Then:
    
    i) All edges $v_1v_2,v_1v_3, v_{n-2}v_n$ and $v_{n-1}v_n$ are edges of $R$,

    ii) For $i=1, 2, \ldots, n-2$, $min_R(v_{i+1}) = max_R(v_i) \geq i+2$, and
    
    iii) For $i=1, 2, \ldots, n-2$, $N^+_R(v_i) = \{v_{min(v_i)}, v_{min(v_i)+1}, \ldots, v_{max(v_i)}\}$. 
\end{lemma} \label{maxpla}

\begin{proof}
\emph{i)} Edges $v_1v_2,v_1v_3, v_{n-2}v_n$ and $v_{n-1}v_n$ do not cross any edge of $T_n$. Therefore, they all must lie in $R$.

\emph{ii)} Clearly $min_R(v_{i+1}) \geq max_R(v_i)$ for $i=1, 2, \ldots, n-2$, otherwise $v_{i+1}v_{min_R(v_{i+1})}$ and $v_iv_{max_R(v_i)}$ would be edges of $R$ that cross each other. Suppose $min_R(v_{t+1}) > max_R(v_t)$ for some integer $1 \leq t \leq n-2$; in particular $v_tv_{min_R(v_{t+1})}$ is not an edge of $R$. Since $R$ is a maximal plane subgraph, edge $v_tv_{min_R(v_{t+1})}$ must cross some edge $v_qv_r$ of $R$. If $q<t<min_R(v_{t+1})<r$, then $q<t<max_R(v_t)<r$ which imples  that edges $v_qv_r$ and $v_tv_{max_R(v_t)}$ of $R$ cross each other, a contradiction. Suppose now $t<q<r<min_R(v_{t+1})$. Notice that $q \neq t+1$, otherwise $v_{t+1}v_r$ is an edge of $R$ with $t+1< r < min_R(v_{t+1})$, which is a contradiction. Then $t+1<q<r<min_R(v_{t+1})$, in which case edges $v_qv_r$ and $v_{t+1}v_{min_R(v_{t+1})}$ cross each other in $R$, once again a contradiction.

By \emph{i)}, $max_R(v_1) \geq 3$. Suppose $max_R(v_j) \geq j+2$ for $j=1, 2, \ldots, t$ and $max_R(v_{t+1}) \leq t+2$ for some integer $1 \leq t<n-2$. Let $l = max \{max_R(v_t), t+3\}$. Notice that $v_{t+1}v_l$ is not an edge of $R$ since $max_R(v_{t+1}) \leq t+2 < l$. As $R$ is a maximal plane subgraph of $T_n$, there is an edge $v_qv_r$ of $R$ that crosses $v_{t+1}v_l$. If $q<t+1<l<r$, then either $v_qv_r$ is an edge of $R$ with $q=t$ and $max_{R(v_t)} < r$ which is a contradiction or $q<t<l<r$ which implies that edges $v_qv_r$ and $v_tv_{max_R(v_t)}$ cross each other in $R$, again a contradiction. Therefore $max_R(v_i) \geq i+2$ for $i=1, 2, \ldots, n-2$.

\emph{iii)} Let $t=1, 2, \ldots, n-2$ and suppose $v_tv_j \notin N^+_R(v_t)$ for some $min_R(v_t)<j<max_R(v_t)$. Since $R$ is a maximal plane subgraph of $T_n$, there is an edge $v_qv_r$ of $R$ that crosses the edge $v_tv_j$. If $q< t < j < r$, then $q< t < min_R(v_t) < r$ and therefore $v_qv_r$ would cross the edge $v_tv_{min(v_t)}$ of $R$, which is not possible. Then $t < q < r < max_R(v_t)$ which implies  $v_qv_r$ crosses the edge $v_tv_{max(v_t)}$ of $R$. Also a contradiction.

\end{proof}

Notice that for any maximal plane subgraph $R$ of $T_n$ with $n \geq 3$, Lemma \ref{maxpla} implies $N^+_R(v_i) \neq \emptyset$ for $i=1, 2, \ldots, n-1$. For any pair $R$, $Q$ of maximal plane subgraphs of $T_n$ let 

\[   
k(R,Q) = 
     \begin{cases}
       max\{t:N_R^+(v_i) = N_Q^+(v_i) \text{ for } i \leq t\}  &\quad\text{ if } N_R^+(v_1) = N_Q^+(v_1) \\ 
       0  &\quad\text{ otherwise } \\
     \end{cases}
\]

and let $m(R,Q) = n-1 - k(R,Q)$.

\medskip

Let $F$ be a set of pairwise non-crossing edges of $T_n$. We denote by $MP(T_n, F)$ the subgraph of $MP(T_n)$ induced by the maximal plane subgraphs of $T_n$ that contain all edges in $F$. The following result  will be used, together with Theorem \ref{teomatchings}, in the proof of our main result. 

\begin{theorem}\label{fijas}
For each integer $n \geq 3$ and each set $F$ of pairwise non-crossing edges of $T_n$, the graph $MP(T_n, F)$ is  connected.
\end{theorem}

\begin{proof}
    Let $F$ be a set of pairwise non-crossing edges of $T_n$ and $R$ and $Q$ be maximal plane subgraphs of $T_n$ with $F \subset E(R) \cap E(Q)$. We shall proof by induction on $m(R,Q)$ that $MP(T_n, F)$ contains a path joining $R$ and $Q$.

    If $m(R,Q) = 0$, then $k(R,Q) = n-1$. This implies $N_R^+(v_i) = N_Q^+(v_i)$ for $i=1, 2, \ldots, n-1$ and therefore, $R = Q$. We proceed assuming that for some integer $t \geq 0$ the result holds whenever $R'$ and $Q'$ are maximal planar subgraphs of $T_n$ with $F \subset E(R') \cap E(Q')$ and $m(R',Q') \leq t $.

    Let $R$ and $Q$ be maximal planar subgraphs of $T_n$ with $F \subset E(R) \cap E(Q)$ and $m(R,Q) = t +1 $. Let $k=k(R,Q)$. 
    
    Then, $N_R^+(v_i) = N_Q^+(v_i) \text{ for } i \leq k$ and $N_R^+(v_{k+1}) \neq N_Q^+(v_{k+1})$; without loss of generality we assume $d_R^+(v_{k+1}) > d_Q^+(v_{k+1})$. Notice that by Lemma \ref{maxpla}, $min_R(v_{k+1}) = max_R(v_{k})= max_Q(v_{k}) = min_Q(v_{k+1})$. Then $max_R(v_{k+1}) > max_Q(v_{k+1})$, and therefore  $v_{k+1}v_{max_R(v_{k+1})}$ is an edge of $R$  which is not an edge of $Q$. Also by Lemma \ref{maxpla}, $max_Q(v_{k+1}) \geq k+3$; this implies $max_R(v_{k+1}) \geq k+4$ and therefore $v_{k+2}v_{max_R(v_{k+1})-1}$ is an edge of $T_n$. Notice that $v_{k+2}v_{max_R(v_{k+1})-1}$ is not an edge of $R$ since it crosses the edge $v_{k+1}v_{max_R(v_{k+1})}$ which does lie in $R$.  Let 
 $$ S_1 = (R - v_{k+1}v_{max_R(v_{k+1})}) + v_{k+2}v_{max_R(v_{k+1})-1}. \text{ See Fig. } \ref{R-S1} \text{ for the case } k =0.$$

\begin{figure}[h] 
\begin{center}
\includegraphics[width=.7\textwidth]{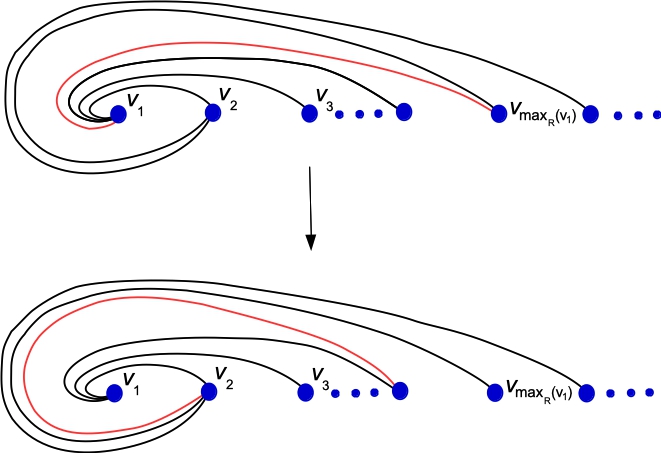}
\label{R-S1} 
\caption{\small Graphs $R$ (above) and $S_1$ (bellow).}
\end{center}
\end{figure}

    We claim $S_1$ is a maximal plane subgraph of $T_n$ adjacent to $R$ in $MP(T_n, F)$. To prove this, assume first that $S_1$ is not a plane graph, then the added edge $v_{k+2}v_{max_R(v_{k+1})-1}$ must cross an edge $v_qv_r$ of $R$, other than the deleted edge $v_{k+1}v_{max_R(v_{k+1})}$. If $k+2 < q < r < max_R(v_{k+1})-1$, then $v_qv_r$ crosses the edge $v_{k+1}v_{max_R(v_{k+1})}$ in $R$ which is not possible. Suppose now $q < k+2  < max_R(v_{k+1})-1 < r$. If $q=k+1$, then $v_qv_r$ is an edge of $R$ with $q=k+1$ and $r \geq  max_R(v_{k+1})$; this implies $v_qv_r$ is the edge $v_{k+1}v_{max_R(v_{k+1})}$ which is not an edge of $S_1$, therefore $q<k+1$. In this case notice that $v_qv_r$ crosses the edge 
    $v_{k+1}v_{max_R(v_{k+1})-1}$ which is an edge of $R$ by Lemma \ref{maxpla}, since $min_R(v_{k+1}) = min_Q(v_{k+1}) \leq max_Q(v_{k+1})  < max_R(v_{k+1})$. Again this is not possible. 

  Next, assume there is an edge $v_qv_r$ of $T_n$, not in $S_1$, such that $S_1 + v_qv_r$ is a plane graph. Since $R$ is a maximal plane subgraph of $T_n$, edge $v_qv_r$ must cross the edge $v_{k+1}v_{max_R(v_{k+1})}$ which is the only edge of $R$ not in $S_1$. If $q < k+1 < max_R(v_{k+1}) < r$, then $v_qv_r$ crosses the edge $v_{k+2}v_{max_R(v_{k+1})-1}$ of $S_1$; and if $ k+1 < q < r < max_R(v_{k+1})$, then $v_qv_r$ would cross the edge $v_{k+2}v_{max_R(v_{k+1})}$ which is an edge of $S_1$ since, by Lemma \ref{maxpla}, $min_R(v_{k+2})=max_R(v_{k+1})$. This ends the proof of our claim. 

    Notice that $F \subset E(S_1)$ since $F \subset E(R) \cap E(Q)$ and $v_{k+1}v_{max_R(v_{k+1})} \in E(R) \setminus E(Q)$; therefore $R$ and $S_1$ are adjacent in $MP(T_n, F)$. 
    
    If $N_{S_1}^+(v_{k+1}) = N_Q^+(v_{k+1})$, then $k(S_1, Q) \geq k+1$ and $m(S_1,Q) \leq t$. By the induction hypothesis $MP(T_n, F)$ contains a path joining $S_1$ and $Q$ which, together with the fact that $R$ and $S_1$ are adjacent in $MP(T_n, F)$, gives a path in $MP(T_n, F)$ that joins $R$ and $Q$.

    If $N_{S_1}^+(v_{k+1}) \neq N_Q^+(v_{k+1})$, then $max_{S_1}(v_{k+1}) > max_Q(v_{k+1})$. We iterate the process defining, one at a time, maximal plane subgraphs $S_2, S_3, \ldots, S_l$ of $T_n$, where $l = max_R(v_{k+1}) - max_Q(v_{k+1})$.

    $$ S_{i+1} = (S_i - v_{k+1}v_{max_{S_i}(v_{k+1})}) + v_{k+2}v_{max_{S_i}(v_{k+1})-1}.$$

    Since $F \subset E(R) \cap E(Q)$ and $v_{k+1}v_{max_{s_i}(v_{k+1})} = v_{k+1}v_{max_R(v_{k+1})-i} \in E(R) \setminus E(Q)$ for $i = 1, 2, \ldots l-1$, $v_{k+1}v_{max_{s_i}(v_{k+1})} \notin F$ for $i = 1, 2, \ldots l-1$. This implies $F\subset E(S_i )$ for $i =1, 2, \ldots, l $. Therefore $R, S_1, S_2, \ldots, S_l$ is a path in $MP(T_n, F)$ joining $R$ and $S_l$. 

    Notice that $N_{S_{i+1}}^+(v_{k+1}) = N_{S_{i}}^+(v_{k+1}) \setminus \{v_{k+1}v_{max_{S_i}(v_{k+1})}\}$ for $i = 1, 2, \ldots, l-1$, therefore $N_{S_l}^+(v_{k+1}) = N_Q^+(v_{k+1})$. Then $k(S_l, Q) \geq k+1$ and $m(S_l, Q) \leq t$. Again by the induction hypothesis, $MP(T_n, F)$ contains a path joining $S_l$ and $Q$ which, together with the path $R, S_1, S_2, \ldots, S_l$, forms a path in $MP(T_n, F)$ that joins $R$ and $Q$.

\end{proof}

\section{Main result}
\label{sec:main}

As shown in Fig. \ref{Nomatch}, not all maximal plane subgraphs of $T_{2m}$ contain perfect matchings. Let $MP_{\cal{M}}(T_{2m})$ denote the subgraph of $MP(T_{2m})$ induced by the set of maximal plane subgraphs of $T_{2m}$ that contain at least one perfect matching.

\begin{figure}[h] 
\begin{center}
\includegraphics[width=.6\textwidth]{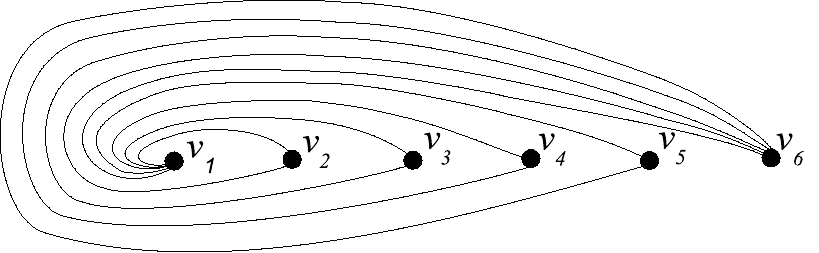}
\label{Nomatch} 
\caption{\small A maximal plane subgraph ot $T_6$ with no perfect matchings.}
\end{center}
\end{figure}

\begin{theorem} \label{mainteo}
For each positive integer $m$, the graph $MP_{\cal{M}}(T_{2m})$  is
connected.
\end{theorem}

\begin{proof}
Let $S$ and $R$ be maximal plane subgraphs of $T_{2m}$, each containing at least one perfect matching,
and let $M(S)$ and $M(R)$ be perfect matchings of $S$ and $R$,
respectively. By Theorem~\ref{teomatchings}, the graph $\mathcal{M}\left( T_{2m}\right) $  is
connected and therefore contains a path
$M(S)=M_{0},M_{1}, \ldots ,M_{k}=M(R)$ joining $M(S)$ and
$M(R)$.

By definition, for $i=1,2,\ldots ,k$, $M_{i}$ is obtained from $M_{i-1}$ by
exchanging the edges of $M_{i}$ and $M_{i-1}$  of an alternating non-crossing cycle with four edges. Therefore,  for
$i=1,2,\ldots ,k$,
$M_{i-1}\cup M_{i}$ is a set of pairwise non-crossing edges of $T_{2m}$ that can be extended to a maximal plane subgraph $S_{i}$ of $T_{2m}$.

Let $S_{0}=S$ and $S_{k+1}=R$. Notice that for $i=0,1,\ldots ,k$, $S_{i}$ and
$S_{i+1}$ are maximal plane graphs of $T_{2m}$ that contain $M_{i}$. By
Theorem~\ref{fijas}, the graph $MP(T_{2m}, M_i)$
contains a path joining $S_{i}$ and $S_{i+1}$. Clearly  $MP(T_{2m}, M_i)$ is a subgraph of $MP_{\cal{M}}(T_{2m})$ for each $i=0,1,\ldots ,k$. Therefore the graph $MP_{\cal{M}}(T_{2m})$ contains a path
joining $S$ and $R$.
\end{proof}

\end{document}